\documentclass[11pt]{article}
\usepackage{vmargin,graphicx,amsmath,amssymb,color,subfigure,enumerate,csquotes, dsfont, mathrsfs}
\usepackage{mathdesign}

\definecolor{webred}{rgb}{0.75,0,0}
\definecolor{webgreen}{rgb}{0,0.75,0}
\usepackage[citecolor=webgreen,colorlinks=true,linkcolor=webred]{hyperref}

\renewcommand{\leq}{\leqslant}
\renewcommand{\geq}{\geqslant}

\newtheorem{theorem}{Theorem}[section]
\newtheorem{lemma}[theorem]{Lemma}
\newtheorem{notation}[theorem]{Notation}
\newtheorem{assumption}[theorem]{Assumption}
\newtheorem{proposition}[theorem]{Proposition}
\newtheorem{corollary}[theorem]{Corollary}

\newtheorem{definition}[theorem]{Definition}

\def\blacksquare{
\thinspace\nobreak \vrule width 5pt height 5pt depth 0pt}
\newtheorem{remark}[theorem]{Remark}
\newenvironment{rem}{\begin{remark} \rm }{\hfill \end{remark}}

\newenvironment{proof}{\begin{trivlist}
      \item[]\hspace{0cm}{\bf Proof:}
      \hspace{0cm}}{\hfill $\blacksquare$
      \end{trivlist}}
\newenvironment{proofof}[1]{\begin{trivlist}
      \item[]\hspace{0cm}{\bf Proof of #1:}
      \hspace{0cm}}{\hfill $\blacksquare$
      \end{trivlist}}
\makeatletter

\@addtoreset{equation}{section}
\makeatother

\def\R{\mathbb R}
\def\Z{\mathbb Z}
\def\S1{\mathbb S^{1}}
\def\Re{{\mathrm {Re}\,}} 
\def\sL{{\rm L}}
\def\sW{{\rm W}}
\def\sH{{\rm H}}
\def\supp{\mathsf{supp}\;}
\def\Sp{\mathsf{sp}}
\def\dx{\,{\rm d}}
\def\re{\mathrm{e}}
\def\dist{\mathsf{dist}}

\def\ri{{\mathsf r}}
\def\le{\ell}
\def\up{\mathsf u}
\def\dow{\mathsf d}
\def\sfA{\mathsf A}
\def\sfS{\mathsf S}
\def\Phir{\Phi_{\ri}}
\def\Phire{\Phi_{\ri,\varepsilon}}
\def\tPhirN{\widetilde\Phi_{\ri,N,h}}
\def\hPhirN{\widehat\Phi_{\ri,N,h}}
\def\Ws{\mathsf W}

\def\Lh{{\mathfrak L}_{h}}
\def\Lr{{\mathfrak L}_{h,\ri}}
\def\Ll{{\mathfrak L}_{h,\le}}
\def\La{{\mathfrak L}_{h,\alpha}}
\def\tLr{{{\mathcal L}}_{h,\ri}}
\def\tLl{{{\mathcal L}}_{h,\le}}

\def\chir{\chi_{\ri}}
\def\chil{\chi_{\le}}
\def\chia{\chi_{\alpha}}
\def\chib{\chi_{\beta}}

\def\cB{{\mathcal B}}
\def\Bri#1{\cB_{\ri}(#1)}
\def\Ble#1{\cB_{\le}(#1)}
\def\Ier{\Bri{\pi-\eta}}
\def\Iel{\Ble{\pi-\eta}}
\def\Iea{\cB_{\alpha}(\pi-\eta)}
\def\I2er{\Bri{\pi-2\eta}}
\def\Ih{I_h}

\def\lar{\lambda(h)} 
\def\BKWr{\psi_{h,\ri}}
\def\BKWl{\psi_{h,\le}}
\def\tphir{\phi_{h,\ri}}
\def\tphil{\phi_{h,\le}}
\def\phir{\varphi_{h,\ri}}
\def\phil{\varphi_{h,\le}}
\def\phia{\varphi_{h,\alpha}}
\def\fr{f_{h,\ri}}
\def\fl{f_{h,\le}}
\def\fa{f_{h,\alpha}}
\def\fb{f_{h,\beta}}
\def\resta{r_{h,\alpha}}

\def\gr{g_{h,\ri}}
\def\gl{g_{h,\le}}
\def\ga{g_{h,\alpha}}
\def\gb{g_{h,\beta}}
\def\gna{{\mathsf g}_{h,\alpha}}
\def\gnb{{\mathsf g}_{h,\beta}}

\begin{document}

\title{Semiclassical tunneling and magnetic flux effects on the circle}
\author{V. Bonnaillie-No\"el\footnote{D\'epartement de Math\'ematiques et Applications (DMA - UMR 8553), PSL, CNRS, ENS Paris, 45 rue d'Ulm, F-75230 Paris cedex 05, France
\texttt{bonnaillie@math.cnrs.fr}},
F. H\'erau\footnote{LMJL - UMR6629, Universit\'e de Nantes, 2 rue de la Houssini\`ere, BP 92208, F-44322 Nantes cedex 3, France, \texttt{frederic.herau@univ-nantes.fr}}
and N. Raymond\footnote{IRMAR - UMR8625, Universit\'e Rennes 1, CNRS, Campus de Beaulieu, F-35042 Rennes cedex, France
\texttt{nicolas.raymond@univ-rennes1.fr}}}
\date{\today}
\maketitle

\begin{abstract}
This paper is devoted to semiclassical tunneling estimates induced on the circle by a double well electric potential in the case when a magnetic field is added. When the two electric wells are connected by two geodesics for the Agmon distance, we highlight an oscillating factor (related to the circulation of the magnetic field) in the splitting estimate of the first two eigenvalues.
\end{abstract}


\section{Introduction and motivations}

\subsection{Motivation}
This paper is devoted to the spectral analysis of the self-adjoint realization of the electro-magnetic Laplacian $(hD_{s}+a(s))^2+V(s)$ on $\sL^2(\S1)$ where the vector potential $a$ and the electric potential $V$ are smooth functions on the circle $\S1$ and where we used the standard notation $D=-i\partial$. In particular we are interested in estimating the spectral gap, in the semiclassical limit, between the first two eigenvalues when the electric potential admits a double symmetric well.
\begin{assumption}\label{V}
In the parametrization $\R\ni s\mapsto \re^{is}\in\S1$, the function $V$ admits exactly two non degenerate minima at $0$ and $\pi$ with $V(0)=V(\pi)=0$ and satisfies $V(\pi-s)=V(s)$.
\end{assumption}

It is well-known that, in dimension one, there is no magnetic field in the sense that the exterior derivative of the $1$-form $a(s) \dx s$ is zero. Nevertheless, since $\S1$ is not simply connected, we cannot gauge out $a$ thanks to an appropriate unitary transform: The circulation of $a$ will remain. This can be explained as follows. Let us define $\varphi(s)=\int_{0}^s \left(a(\sigma)-\xi_{0}\right) \dx \sigma$ with $\xi_{0}=\int_{-\pi}^\pi a(\sigma) \dx \sigma$ so that $\varphi$ is well-defined and smooth on $\S1$. Then let us consider the conjugate operator
\begin{align*}
\Lh&=\re^{i\varphi/h}\left[(hD_{s}+a(s))^2+V(s)\right] \re^{-i\varphi/h}\\
&=(hD_{s}+a(s)-\varphi'(s))^2+V(s)\\
&=(hD_{s}+\xi_{0})^2+V(s).
\end{align*}
The aim of this paper is to investigate the effect of the circulation $\xi_{0}$ of $a$ on the semiclassical spectral analysis.

\subsection{Results}
The analysis of this paper gives an asymptotic result of the splitting between the first two eigenvalues $\lambda_1(h)$ and $ \lambda_2(h)$ of $\Lh$, when the potential $V$ has some symmetries.
\begin{theorem}\label{th.gap}
Let $\kappa$ be the geometric constant defined by
\begin{equation} \label{lambda0}
\kappa = \sqrt{\frac{V''(0)}{2}}.
\end{equation}
Then, as soon as $h$ is small enough, there are only two eigenvalues of $\Lh$ in the interval $\Ih = (-\infty, 2\kappa h)$ and they both satisfy
$$\mbox{ for }j=1,2,\qquad\lambda_j(h) = \kappa h + o(h)\quad\mbox{ as }h\to0.$$
Let us define the (positive) Agmon distances
$$
\sfS_{\up} = \int_{[0,\pi]} \sqrt{V(\sigma)} \dx \sigma, \qquad
\sfS_{\dow} = \int_{[0, -\pi]} \sqrt{V(\sigma)} \dx \sigma, \qquad\mbox{ and}\qquad
\sfS = \min \{ \sfS_{\up}, \sfS_{\dow} \},
$$
and the two constants
$$
\sfA_{\up} = \exp\left( - \int_{[0,\frac\pi2]} \frac{\partial_\sigma {\sqrt{V}} - \kappa}{\sqrt{V}} \dx\sigma\right), \qquad
\sfA_{\dow} = \exp\left( \int_{[-\frac\pi2, 0]} \frac{\partial_\sigma {\sqrt{V}} + \kappa}{\sqrt{V}} \dx\sigma\right).
$$
Then we have the spectral gap estimate
\begin{equation}\label{eq.tun}
\lambda_2(h) - \lambda_1(h) = 2| w_{0}(h)| + h^{3/2} {\cal O}(\re^{-\sfS/h}) ,
\end{equation}
with
\begin{equation}\label{eq.woh}
w_{0}(h)=2h^{1/2}\sqrt{ \frac{\kappa}{\pi}}
\left( \sfA_{\up} \sqrt{V\left(\frac\pi2\right)} \re^{\frac{i\xi_{0} \pi-\sfS_{\up}}h}
+\sfA_{\dow} \sqrt{V\left(-\frac{\pi}2\right)} \re^{\frac{-i\xi_{0} \pi-\sfS_{\dow}}h} \right).
\end{equation}
\end{theorem}
\begin{rem}
The constants $\sfS_{\up}$ and $\sfS_{\dow}$ correspond to integrations in the upper and respectively lower part of the circle for the Agmon distance. Then two situations may occur:
\begin{enumerate}
\item If the two Agmon distances $\sfS_{\up}$ and $\sfS_{\dow}$ are different, only one term in the sum \eqref{eq.woh} defining $w_{0}(h)$ is predominent and $w_{0}(h)$ is not zero for $h$ small enough. In this case, there exists a unique geodesic linking the two wells,  corresponding either to the upper part of the circle, or to the lower part. Moreover, the circulation $\xi_{0}$ is not involved in the estimate of the tunneling effect: we get an estimate similar to what happens in the purely electric situation (see \cite{Har80, Rob87} and more generally \cite{Simon84, HelSj84, HelSj85}).
\item If $\sfS_{\up} = \sfS_{\dow}$, the situation is completely different: due to the circulation, the interaction term $w_{0}(h)$ can vanish for some parameters $h$ and the eigenvalues can be equal up to an error of order ${\cal O}(h^{3/2} \re^{-\sfS /h})$. This corresponds to a crossing (up to the forementionned error) of these first two eigenvalues. Note that this does not mean that the eigenvalues $\lambda_{1}(h)$ and $\lambda_{2}(h)$ effectively cross but the gap is in ${\cal O}(h^{3/2} \re^{-\sfS /h})$.
\end{enumerate}
\end{rem}
When the potential $V$ is even, we are in the second situation and we have
$$
\sfA_{\up} =\sfA_{\dow} = \sfA , \qquad \sfS_{\up} = \sfS_{\dow} = \sfS , \qquad V\left(-\frac\pi2\right) = V\left(\frac\pi2\right),
$$
and we immediately deduce the following splitting estimate.
\begin{theorem}\label{th.gapsym}
Assume that $V$ is even, then
$$
\lambda_2(h) - \lambda_1(h) = 8 h^{1/2} \sfA \sqrt{V\Big(\frac\pi2\Big)} \sqrt{\frac{\kappa}{\pi}} \left|\cos\left( \frac{\xi_{0} \pi}h\right) \right| \re^{-\sfS /h}+ h^{3/2} {\cal O}\left(\re^{-\sfS /h}\right).
$$
\end{theorem}

\subsubsection*{Organization of the paper and strategy of the proofs}
In order to prove Theorem \ref{th.gap}, we will follow the strategy developed by Helffer and Sj\"ostrand in \cite{HelSj84, HelSj85} (see also the lecture notes by Helffer \cite[Section 4]{Hel88}) for the pure electric case. Thanks to a change of gauge, the investigation of the present paper can be reduced to the electric case only locally and not globally due to the circulation $\xi_{0}$. In Section \ref{Sec.simplewell}, we recall the WKB approximations of the first eigenfunction in the simple well case. In Section \ref{Sec.doublewell}, we explain how we can construct a $2$ by $2$ Hermitian matrix (the so-called \enquote{interaction matrix}) from the eigenfunctions of each well, which describes the splitting of first two eigenvalues of $\Lh$. This strategy is well-known (see for instance \cite{DiSj99} for a short presentation and \cite{Har80} for a complete description of the main terms) and is given here for completeness. The aim of the present paper is to highlight its oscillatory consequences on the interaction term in the non zero circulation case. To authors' knowledge this strategy was never applied in this context and the understanding of this model might be a main step towards the estimate of the pure magnetic tunnel effect in higher dimension (see \cite{HelSj87} and our recent contribution \cite[Section 5.3]{BHR14}). Note here that the influence of the circulation on the first eigenvalue has also been analyzed in \cite[Theorem 7.2.2.1]{Hel88} when $V$ admits a unique and non degenerate minimum. This question was also tackled by Outassourt in \cite{Out87} in a periodic framework. Finally, in Section \ref{Sec.Comput}, we analyze the semiclassical behavior of the interaction matrix in terms of the WKB approximations.

\section{Simple well cases}\label{Sec.simplewell}
In this section we study simple well configurations. First, we consider the well $s=0$. In the last part, we explain how we can transfer what was done for the well $s=0$ to the well $s=\pi$ thanks to a unitary transform.

\subsection{Local reduction to the pure electric situation}
Let us introduce the Dirichlet realization attached to the well $s=0$. For any $\rho\in (0,\pi]$, we define
$$ \Bri{\rho}:=\cB(0,\rho)= (-\rho,\rho).$$
Given $\eta>0$, let us consider $\Lr$ the Dirichlet realization of $(hD_{s}+\xi_{0})^2+V(s)$ on the space $\sL^2(\Ier, \dx s)$. Since $\Ier$ is simply connected, we can perform a gauge transform so that the study of $\Lr$ is reduced to the one of the operator
\begin{equation}\label{eq.jauge}
\tLr = \re^{\frac{i\xi_{0}s}{h}}\Lr \re^{\frac{-i\xi_{0}s}{h}} = h^2 D_{s}^2+V(s),
\end{equation}
defined on $\mathsf{Dom}\left(\tLr\right)=\sH^2(\Ier)\cap \sH^1_{0}(\Ier)$. Let us denote by $\lar$ the ground state energy of $\tLr$ and $\tphir$ the positive and $\sL^2$-normalized eigenfunction of $\tLr$ associated with the lowest eigenvalue $\lar$. We have
$$\tLr \tphir = \left(h^2 D_{s}^2 + V\right) \tphir = \lar \tphir\qquad \mbox{ on }\Ier.$$
Then, by gauge tranform, the function defined on $\Ier$ by
\begin{equation}\label{def.phir}
\phir(s) = \re^{-i\frac{\xi_{0} s}h} \tphir(s),
\end{equation}
is a $\sL^2$-normalized eigenfunction of $\Lr$ associated with $\lar$. \\

In the next section, we recall some results about the WKB analysis of the operator $\tLr$. In Section \ref{simpleagmon} we recall Agmon estimates and in particular prove the exponential decay of eigenfunctions. In the following subsection, we establish uniform estimates of the difference between the eigenfunctions and the WKB quasimodes.

\subsection{WKB approximations in a simple well}
This section is devoted to recall the structure of the first WKB quasimode of $\tLr$.
\begin{lemma} \label{lem.QM}
The asymptotic WKB series for the first quasimode of $\tLr$ is given by
 \begin{equation}\label{Ansatz-WKB}
\BKWr=\chir \Psi_{h,\ri},\quad \mbox{ with }\quad \Psi_{h,\ri}(s)= h^{-1/4} \re^{-\frac{\Phir(s)}{h}}\ \sum_{j\geq0}h^j a_{j}(s),\quad \forall s\in\Bri{\pi}, 
\end{equation}
where
\begin{enumerate}[i)]
\item $\chir$ is a smooth cut-off function supported on $\Ier$ with $0\leq \chir\leq 1$ and $\chir = 1$ on $\I2er$,
\item $\Phir $ is the standard Agmon distance to the well at $s=0$:
\begin{equation}\label{Phi0}
\Phir(s)=\int_{[0,s]} \sqrt{V(\sigma)} \dx \sigma, \qquad\forall s\in \Bri{\pi},
\end{equation}
\item $a_{0}$ is a solution of the associated transport equation
\begin{equation} \label{transport}
\Phir'\,\partial_{s}a_{0}+\partial_{s}\left(\Phir'\,a_{0}\right) =\kappa a_{0},
\end{equation}
with $\kappa$ defined in \eqref{lambda0}. It can be given explicitly by
$$
a_0(s) =\left(\frac{\kappa}{\pi}\right)^{1/4} \exp\left(- \int_0^s \frac{\Phir''(\sigma) - \kappa}{ 2 \Phir'(\sigma)} \dx\sigma\right),\qquad \forall s\in\Bri\pi.
$$
\end{enumerate}
The function $\BKWr$ is a $\sL^2$-normalized WKB quasimode in the sense that
\begin{equation}\label{Lr-lar-ter}
\re^{\Phir/h}\left(\tLr-\mu_{\ri}(h)\right)\BKWr = {\cal O}(h^\infty) \qquad \mbox{ in } \sL^2(\I2er),
\end{equation}
where $\mu_{\ri}(h)$ is the first quasi-eigenvalue given by the asymptotic series
$$
\mu_{\ri}(h) =\kappa h +\sum_{j\geq 2}\mu_{\ri,j}h^j.
$$
Moreover, we have
\[
\partial_s\BKWr (s)= - h^{-5/4} \Phir'(s)
\re^{-\frac{\Phir(s)}{h}} a_{0}(s)(1+ {\cal O}(h)), \qquad\forall s \in \I2er.
\]
\end{lemma}
\begin{proof}
The proof of the result is classical (see \cite{Har80,Rob87}) and we just recall the computation of $a_{0}$, which is quite easy since we are in dimension one. For $s \in \Ier$, we check that
$$
V(s) = \kappa ^2 s^2 + {\cal O}(s^3) \qquad \mbox{ and}\qquad \Phir (s) = \kappa\frac{s^2}{2} + {\cal O}(s^3).
$$
Solving the transport equation \eqref{transport}, we get
$$
a_0(s) = K_0  \exp\left(- \int_0^s \frac{\Phir''(\sigma) - \kappa}{ 2 \Phir'(\sigma)} \dx\sigma\right),
$$
where $K_{0}$ is a normalization constant determined by
$$
1 = \int_{\Ier} \big|\BKWr(s)\big|^2\dx s = K_0^2 h^{-1/2} \int_\R \re^{- \kappa s^2/h} \dx s(1 + {\cal O}(h)) = K_0^2\sqrt{\frac{\pi}{\kappa}}+ {\cal O}(h).
$$
Thus $K_0 = ({\kappa/\pi})^{1/4}$.
\end{proof}
The explicit form of the quasimode will be used for the computation of the splitting between the first two eigenvalues of $\Lh$ in Section~\ref{Sec.Comput}.

\subsection{Agmon estimates and WKB approximation} \label{simpleagmon}
Let us recall the following lemma (see \cite{Ray13} for a close version) which will be useful to prove localization estimates.
\begin{lemma}\label{localization}
Let $\sH$ be a Hilbert space and $P$ and $Q$ be two unbounded and symmetric operators defined on a domain $\mathsf{D}\subset \sH$. We assume that $P(\mathsf{D})\subset \mathsf{D}$, $Q(\mathsf{D})\subset \mathsf{D}$ and $[[P, Q],Q]=0$ on $\mathsf{D}$.
Then, for $u\in \mathsf{D}$, we have
\[
\Re\langle Pu, P Q^2u \rangle=\|PQu\|^2-\|[Q,P]u\|^2.
\]
\end{lemma}
This lemma will be applied with $P$ the derivation and $Q$ the multiplication by a smooth function.

With the aim of proving that our Ansatz is a good approximation of the first eigenfunction $\tphir $ of $\tLr$, we first establish some Agmon estimates.
\begin{proposition}\label{Agmon}
Let $\Phi$ be a Lipschitzian function such that
\begin{equation}\label{eq.Phi'glo}
V(s)-|\Phi'(s)|^2\geq 0,\qquad\forall s\in\Ier,
\end{equation}
and let us assume that there exist $M>0$ and $R>0$ such that for all $h\in(0,1)$,
\begin{align}
V(s)-|\Phi'(s)|^2\geq M h,&\qquad \forall s\in\Ier\cap\complement \Bri{Rh^{1/2}},\label{eq.Phi'}\\
 |\Phi(s)|\leq Mh,&\qquad \forall s\in\Bri{Rh^{1/2}}.\label{eq.Phibd}
\end{align}
Then, for all $C_{0}\in(0,M)$, there exist positive constants $c, C$ such that, for $h\in (0,1)$, $z\in[0,C_{0}h]$, $u\in\mathsf{Dom}\left(\tLr\right)$,
\begin{equation}\label{Agmon-L2}
ch\| \re^{\Phi/h}u\|_{\sL^2(\Ier)}\leq\| \re^{\Phi/h}(\tLr -z)u\|_{\sL^2(\Ier)}+Ch\|u\|_{\sL^2(\Ier\cap \Bri{Rh^{1/2}})},
\end{equation}
and
\begin{equation}\label{Agmon-H1}
\left\| h D_{s} \left(\re^{\Phi/h} u\right)\right\|_{\sL^2(\Ier)}^2
\leq {\frac{C}{h}} \| \re^{\Phi/h}(\tLr -z)u\|_{\sL^2(\Ier)}^2+ Ch\|u\|_{\sL^2(\Ier\cap \Bri{Rh^{1/2}})}^2.
\end{equation}
\end{proposition}

\begin{proof}
We apply Lemma \ref{localization} with $P = hD_s$, $Q= \re^{\Phi/h}$ and $u \in\mathsf{Dom}\left(\tLr\right)$ to get
\begin{multline*}
\Re\left(\int_{\Ier} hD_s u\, hD_{s}\left(\re^{2\Phi/h}u\right)\dx s\right)\\
= \int_{\Ier} |h D_{s}(\re^{\Phi/h}u)|^2\dx s - \int_{\Ier} |\Phi'(s)|^2 \re^{2\Phi/h} |u|^2\dx s.
\end{multline*}
Integrating by parts, adding the electric potential $V$, and recalling that $\tLr = h^2D_s^2 + V$, we find
\begin{align*}
\int_{\Ier} |h D_{s}(\re^{\Phi/h}u)|^2\dx s&+\int_{\Ier} (V(s)-|\Phi'(s)|^2) \re^{2\Phi/h} |u|^2\dx s\\
&= \Re\left(\int_{\Ier}\tLr u\, \re^{2\Phi/h}u \dx s\right)
\leq\| \re^{\Phi/h}\tLr u\|\| \re^{\Phi/h}u\|.
\end{align*}
Using \eqref{eq.Phi'glo} and \eqref{eq.Phi'}, we get
$$\int_{\Ier} |hD_{s}(\re^{\Phi/h}u)|^2\dx s+ Mh\int_{\Ier\cap\complement \Bri{Rh^{1/2}}}\re^{2\Phi/h}|u|^2\dx s\leq\| \re^{\Phi/h}\tLr u\|\| \re^{\Phi/h}u\|.$$
Thanks to \eqref{eq.Phibd}, $\Phi/h$ is uniformly bounded with respect to $h$ on $\Bri{Rh^{1/2}}$ and we deduce
$$\|h D_{s}(\re^{\Phi/h}u)\|^2+Mh\| \re^{\Phi/h}u\|^2\leq\| \re^{\Phi/h}\tLr u\|\| \re^{\Phi/h}u\|+C_{R}h\|u\|^2_{\sL^2(\Ier\cap \Bri{Rh^{1/2}})}.$$
For $|z|\leq C_{0}h$, we get
\begin{multline}\label{eq:Agmon}
\|hD_{s}(\re^{\Phi/h}u)\|^2+(M-C_{0})h\| \re^{\Phi/h}u\|^2\\
\leq\| \re^{\Phi/h}(\tLr -z)u\|\| \re^{\Phi/h}u\|+C_{R}h\|u\|^2_{\sL^2(\Ier\cap \Bri{Rh^{1/2}})}.
\end{multline}
Since $C_{0}<M$, this gives \eqref{Agmon-L2}. Then we combine \eqref{eq:Agmon} with \eqref{Agmon-L2} to get \eqref{Agmon-H1}.
\end{proof}
\begin{proposition}\label{prop.poids}
Let $c_{0}>0$ such that
\begin{equation}\label{eq.c0}
V(s)\geq c_{0}s^2\qquad \mbox{ and}\qquad \Phir(s)\geq c_{0}s^2,\qquad\forall s\in\Ier.
\end{equation}
Proposition \ref{Agmon} applies in the following cases:
\begin{enumerate}[ \rm(a)]
\item\label{poids1} for $\varepsilon \in(0,1)$, the rough weight $\Phire=\sqrt{1-\varepsilon }\Phir $ with $R>0$ and $M=c_{0}\varepsilon R^2$,
\item\label{poids2} for $N\in\mathbb{N}^*$ and $h\in(0,1)$, the precised weight $\tPhirN =\Phir - N h \ln \left(\max \left( \frac{\Phir}{h}, N\right)\right)$, with $R=\sqrt{\frac{N}{c_{0}}}$ and $M=N\inf_{\Ier}\frac{V}{\Phir}$,
\item\label{poids3} for $\varepsilon \in(0,1)$, $N\in\mathbb{N}^*$ and $h\in(0,1)$, the intermediate weight
\begin{equation}\label{eq.poids3}
\hPhirN (s) = \min \left\{\tPhirN (s), \sqrt{1-\varepsilon }\displaystyle{\inf_{t \in \supp\chi'_{\ri}} \left(\Phir(t) +\int_{[s,t]}\sqrt{V(\sigma)}\dx \sigma\right)}\right\},
\end{equation}
with $R=\sqrt{\frac{N}{c_{0}}}$ and $M=N\min\left(\varepsilon ,\inf_{\Ier}\frac{V}{\Phir}\right)$, where we recall that $\chir'$ is supported in $\Ier \setminus \I2er$.
\end{enumerate}
\end{proposition}
\begin{proof}
Note that the existence of $c_{0}>0$ is guarranted since the function $V$ admits a unique and non degenerate minimum on $\Ier$ at $0$. Using the definition \eqref{Phi0} of $\Phir$, we have directly \eqref{eq.Phibd} for $\Phir$ and consequently for the other weights $\tPhirN$ and $\hPhirN$ which are smaller. Let us now prove \eqref{eq.Phi'glo} and \eqref{eq.Phi'} for each choice.
\begin{enumerate}[\rm (a)]
\item We have $V-|\Phire'|^2=\varepsilon V$. Combining this with the positivity of $V$ or \eqref{eq.c0} gives \eqref{eq.Phi'glo} and \eqref{eq.Phi'}.
\item On $\{\Phir <Nh\}$, we have $|\tPhirN'|^2=|\Phir'|^2=V$. \\
On $\{\Phir \geq Nh\}$, we get
$$\tPhirN'=\Phir'\left(1-\frac{Nh}{\Phir}\right),$$
so that
\begin{equation}\label{eq.vitem2}
V-|\tPhirN'|^2=V\frac{Nh}{\Phir}\left(2-\frac{Nh}{\Phir}\right)\geq Nh\frac{V}{\Phir}\geq cNh\geq 0,
\end{equation}
since the function ${V}/{\Phir}$ is continuous and bounded from below by some $c>0$ on $\Ier$. This proves \eqref{eq.Phi'glo}. According to \eqref{eq.c0}, for all $R>0$ and $h\in(0,1)$, we have $\Phir \geq c_{0}R^2h$ on $\Ier\cap\complement \Bri{Rh^{1/2}}$. In particular, for $R\geq R_{0}=\sqrt{{N}/{c_{0}}}$, we get
$$\Ier\cap\complement \Bri{Rh^{1/2}}\subset \{\Phir \geq Nh\}.$$
Recalling \eqref{eq.vitem2}, this establishes \eqref{eq.Phi'}.
\item We notice that the infimum in the definition of $\hPhirN$ is a minimum. Thus, almost everywhere on $\Ier$, we have either $|\hPhirN'|=\sqrt{1-\varepsilon }\sqrt{V}$, or $|\hPhirN'|=|\tPhirN'|$. Then we apply  Proposition~\ref{prop.poids} \eqref{poids1} and \eqref{poids2}.
\end{enumerate}
\end{proof}
\begin{remark}
The weights introduced in Proposition \ref{prop.poids} are essential to prove that the eigenfunctions of $\tLr$ are approximated by their WKB expansion in the space $\sL^2(\re^{\Phir/h}ds)$ (as we will see in Proposition \ref{uniform}). The rough weight $\Phire=\sqrt{1-\varepsilon }\Phir$ would not be enough to get the main term of the tunneling estimate \eqref{eq.tun}. The precised weight $\tPhirN$ is introduced to get an approximation of the eigenfunctions in the space $\sL^2(h^{-N}\re^{\Phir/h}\dx s)$ with a fixed and large $N\in\mathbb{N}$; the factor $h^{-N}$ will be absorbed since the approximation is valid modulo ${\cal O}(h^{\infty})$. The intermediate weight  $\hPhirN$ is only a slight modification of $\tPhirN$ (see Lemma \ref{lem.hPhirN}) on $\complement K$ where the weight $\tPhirN$ becomes bad.
\end{remark}
We end this section with some properties, which will be used later, about the weight $\hPhirN$ defined in \eqref{eq.poids3}.
\begin{lemma}\label{lem.hPhirN}
Let $K$ be a compact with $K\subset\I2er$. We consider the weight defined in Proposition \ref{prop.poids} \eqref{poids3}. For all $N\in\mathbb{N}^*$, there exists $\varepsilon_{0}$ such that for all $0<\varepsilon < \varepsilon_0$,  there exist $h_{0}>0$ and $R>0$ such that, for all $h\in(0,h_{0})$, we have
\begin{enumerate}[\rm (1)]
\item\label{eq.lmhPhi1} $\hPhirN \leq \Phir$ on $\Ier$,
\item\label{eq.lmhPhi3} $\hPhirN =\tPhirN$ on $K$,
\item\label{eq.lmhPhi2} $\hPhirN = \sqrt{1-\varepsilon }\Phir$ on $\supp \chir'$.
\end{enumerate}
\end{lemma}
\begin{proof}
\begin{enumerate}[\rm (1)]
\item The first inequality comes immediately from the definition of $\hPhirN$.
\item By continuity and since  $K$ and  the complementary of $\I2er$ are disjoint compacts, there exists $\varepsilon_0$ such that for all $0 < \varepsilon <\varepsilon_0$ and for all $s \in K$,
$$\tPhirN(s)\leq \Phir(s)\leq \sqrt{1-\varepsilon}\displaystyle{\inf_{t \in \supp\chi'_{\ri}} \left(\Phir(t) +\int_{[s,t]}\sqrt{V(\sigma)}\dx \sigma\right)}.$$
By definition of $\hPhirN$, we deduce that $\hPhirN = \tPhirN$ on $K$.
\item Let us now consider $s\in\supp\chir'$. There exists $h_{0}>0$ (depending on $\varepsilon$) such that for all $h\in(0,h_{0})$, we have
\[\begin{cases}
\inf_{t \in \supp\chi'_{\ri}} \left(\Phir(t) +\int_{[s,t]}\sqrt{V(\sigma)}\dx \sigma\right) &= \Phir(s),\\[5pt]
\tPhirN (s) =\Phir(s) + {\cal O}(h\ln h) &\geq \sqrt{1-\varepsilon }\Phir(s).
\end{cases}\]
Thus $\hPhirN  = \sqrt{1-\varepsilon }\Phir $ on $\supp\chir'$.
\end{enumerate}
\end{proof}

\subsection{Weighted comparison between quasimodes and eigenfunctions}
We may now provide the approximation of $\tphir $ by the WKB construction $\BKWr$ defined in \eqref{Ansatz-WKB}. Let us introduce the projection
$$\Pi_{\ri}\psi=\langle\psi, \tphir \rangle \tphir .$$
\begin{proposition} \label{uniform}
Let $K$ be a compact set with $K\subset\I2er$. We have both in the $\sL^\infty(K)$ and in the $\sL^2(K)$ sense
\begin{align}
  \re^{\Phir/h}\left(\BKWr -\Pi_{\ri}\BKWr \right) &= {\cal O}(h^{\infty}), \label{eq.unif1}\\
  \re^{\Phir/h} D_{s}\left(\BKWr -\Pi_{\ri}\BKWr \right) &= {\cal O}(h^{\infty}). \label{eq.unif2}
\end{align}
\end{proposition}
\begin{proof}
Let us apply Proposition \ref{Agmon} with $u=\BKWr-\Pi_{\ri}\BKWr$ and $z=\lar$ and the weight $\Phi=\hPhirN$ defined in Proposition \ref{prop.poids} \eqref{poids3}. We get
\begin{multline}\label{eq.profProp251}
ch\| \re^{\hPhirN /h}u\|^2_{\sL^2(\Ier)}+\left\| h D_{s} \left(\re^{\hPhirN /h} u\right)\right\|_{\sL^2(\Ier)}^2\\
 \leq Ch^{-1}\| \re^{\hPhirN /h}(\tLr -\lar)\BKWr\|_{\sL^2(\Ier)}^2+ Ch\|u\|_{\sL^2(\Ier\cap \Bri{Rh^{1/2}})}^2.
\end{multline}
Let us investigate the first term in the r.h.s. of \eqref{eq.profProp251}.
Using Lemma~\ref{lem.QM}, we have, in the sense of differential operators,
\begin{align}
  \re^{\hPhirN /h}(\tLr &-\lar)\BKWr=\re^{\hPhirN /h}(\tLr -\lar)\chir\Psi_{h,\ri}\nonumber\\
  &=\re^{\hPhirN /h}\chir(\tLr -\lar)\Psi_{h,\ri}+\re^{\hPhirN /h}[\tLr,\chir]\Psi_{h,\ri}\nonumber\\
  &=\re^{(\hPhirN -\Phir )/h}{\cal O}_{\sL^\infty(\Ier)}(h^\infty)+\re^{(\hPhirN -\Phir )/h}{\cal O}_{\sL^\infty(\supp\chir')}(1).\label{eq.termPropo26}
\end{align}
Using Lemma~\ref{lem.hPhirN}, there exists $c_{1}>0$ such that
\begin{align*}
  \re^{(\hPhirN -\Phir )/h}{\cal O}_{\sL^\infty(\Ier)}(h^\infty)&={\cal O}_{\sL^\infty(\Ier)}(h^\infty),\\
  \re^{(\hPhirN -\Phir )/h}=\re^{-(1-\sqrt{1-\varepsilon })\Phir/h}\leq \re^{-c_{1}/h} &={\cal O}(h^{\infty})\quad\mbox{ on }\supp\chir'.
\end{align*}
Putting these estimates in \eqref{eq.termPropo26}, we deduce that
\begin{equation} \label{eq.lhrlambda}
Ch^{-1}\| \re^{\hPhirN /h}(\tLr -\lar)\BKWr\|_{\sL^2(\Ier)}^2 = {\cal O}(h^\infty).
\end{equation}
Let us deal with the second term in the r.h.s. of \eqref{eq.profProp251}.
By definition, $\Pi_{\ri}\BKWr $ belongs to the kernel of $\tLr-\lar$ and, since the gap between the lowest eigenvalues of $\tLr$ is of order $h$, the spectral theorem proves that there exists $c>0$ such that
\begin{multline}\label{eq.thsp}
ch\| u \|_{\sL^2(\Ier)} =ch\|\BKWr -\Pi_{\ri}\BKWr \|_{\sL^2(\Ier)}\\
\leq\left\|\left(\tLr -\lar\right) u\right\|_{\sL^2(\Ier)}
=\left\|\left(\tLr -\lar\right)\BKWr\right\|_{\sL^2(\Ier)} = {\cal O}(h^\infty),
\end{multline}
where we have used \eqref{Lr-lar-ter} for the last estimate.\\
Consequently \eqref{eq.profProp251} becomes
\begin{equation}\label{eq.hDSexpu}
ch\| \re^{\hPhirN /h}u\|^2_{\sL^2(\Ier)}+\left\| h D_{s} \left(\re^{\hPhirN /h} u\right)\right\|_{\sL^2(\Ier)}^2
 =  {\cal O}(h^\infty).
 \end{equation}
By Sobolev embedding, we deduce that, as well as in $\sL^\infty(\Ier)$ as in $\sL^2(\Ier)$,
$$
 h\re^{\hPhirN /h} u =  {\cal O}(h^\infty).
$$
To deduce \eqref{eq.unif1}, we first recall Lemma \ref{lem.hPhirN} \eqref{eq.lmhPhi3}, so that $\hPhirN = \tPhirN$ on $K$. Then we have, in $\sL^\infty(K)$ and in $\sL^2(K)$,
\begin{equation} \label{eq.linfestimate}
 h\re^{\tPhirN /h} u =  {\cal O}(h^\infty).
\end{equation}
Now the definition of $\tPhirN$ (given in Proposition \ref{prop.poids} \eqref{poids2}) implies that in $\sL^\infty(K)$ we have
\begin{equation} \label{eq.differencetphiphi}
  \re^{(\Phir-\tPhirN)/h} = {\cal O}(h^{-N}).
\end{equation}
By using \eqref{eq.linfestimate}, we get, in $\sL^\infty(K)$ and in $\sL^2(K)$,
\begin{equation} \label{eq.linfestimate2}
\re^{\Phir /h} u = h^{-1}{\cal O}(h^{-N}) {\cal O}(h^\infty) =  {\cal O}(h^\infty).
\end{equation}
This proves \eqref{eq.unif1}.\\
Now we deal with the $\sL^2(K)$ estimate in \eqref{eq.unif2}. Let us recall that Lemma~\ref{lem.hPhirN} \eqref{eq.lmhPhi3} gives
\begin{equation}\label{eq.PhirtPhirK}
\hPhirN = \tPhirN\qquad \mbox{ on }K.
\end{equation}
We first write that
\begin{equation}\label{eq.hDSexpu5}
\left\| \re^{\hPhirN /h}h D_{s} u\right\|_{\sL^2(K)}^2
\leq \left\| h D_{s} \left(\re^{\hPhirN /h} u\right)\right\|_{\sL^2(K)} +  \left\| \tPhirN'\left(\re^{\hPhirN /h} u\right)\right\|_{\sL^2(K)}.
\end{equation}
Using that $|\tPhirN'|^2\leq V$ which is bounded and \eqref{eq.hDSexpu}, we deduce, by Lemma \eqref{lem.hPhirN} \eqref{eq.lmhPhi3}, 
\begin{equation}\label{eq.hDSexpu6}
\left\| \re^{\tPhirN /h} h D_{s} u \right\|_{\sL^2(K)} =\left\| \re^{\hPhirN /h} h D_{s} u \right\|_{\sL^2(K)}^2 = {\cal O}(h^\infty).
\end{equation}
Next using \eqref{eq.differencetphiphi}, we have the desired $\sL^2(K)$ estimate in \eqref{eq.unif2}:
\begin{equation}\label{eq.hDSexpu8}
\left\| \re^{\Phir /h} h D_{s} u \right\|_{\sL^2(K)} = {\cal O}(h^\infty).
\end{equation}
As a complementary result and for further use, let us do a new  commutation with $hD_s$. We have
\begin{equation*}
\| h D_{s}(\re^{\Phir/h} u) \|_{\sL^2(K)} \leq
\|\re^{\Phir/h} h D_{s} u \|_{\sL^2(K)} + \| \Phir' \re^{\Phir/h} u  \|_{\sL^2(K)}.
\end{equation*}
Using \eqref{eq.unif1} in $\sL^2(K)$, the fact that $|\Phir'|^2=V$, $V$ is bounded and \eqref{eq.hDSexpu8}, we infer
\begin{equation}\label{eq.hDSexpu9}
\left\| h D_{s} \left(\re^{\Phir /h} u \right) \right\|_{\sL^2(K)} = {\cal O}(h^\infty).
\end{equation}
We end up with the $\sL^\infty(K)$ estimate in  \eqref{eq.unif2}. From \eqref{eq.lhrlambda} restricted to $K$, \eqref{eq.PhirtPhirK} and \eqref{eq.differencetphiphi}, we have
\begin{equation} \label{eq.lhrlambda2}
\| \re^{\Phir /h}(\tLr -\lar)\BKWr\|_{\sL^2(K)}^2 = {\cal O}(h^\infty).
\end{equation}
Since $\Pi_{\ri}\BKWr$ is an eigenfunction, we get
$$
\left\| \re^{\Phir/h}\left(\tLr-\lar\right) u\right\|_{\sL^2(K)}^2
=\left\| \re^{\Phir/h}\left(\tLr-\lar\right)(\BKWr -\Pi_{\ri}\BKWr )\right\|_{\sL^2(K)}^2
= {\cal O}(h^\infty).
$$
By definition of $\tLr$, this provides
$$\left\| \re^{\Phir/h}\left(h^2D_{s}^2+V(s)-\lar\right)u\right\|_{\sL^2(K)}^2 = {\cal O}(h^\infty).$$
Thanks to \eqref{eq.unif1} in $\sL^2(K)$ and since $\lar=\mathcal{O}(h)$ and $V$ is bounded, we infer
\begin{equation} \label{eq.Dsint}
\left\| \re^{\Phir/h}h^2D_{s}^2u\right\|_{\sL^2(K)}= {\cal O}(h^\infty).
\end{equation}
We have
\begin{align} \label{eq.sumDS2}
 (h^2D_{s}^2) \left(\re^{\Phir/h}u\right)
& =\re^{\Phir/h}(h^2D_{s}^2)u + \left[h^2 D_{s}^2,\re^{\Phir/h}\right]u,
\end{align}
where
\begin{equation}\label{eq.commut}
\left[h^2 D_{s}^2,\re^{\Phir/h}\right]u
= -\re^{\Phir/h}\left(2ih \Phir' D_{s}u +|\Phir'|^2 u +h \Phir'' u \right).
\end{equation}
Since $\Phir'$ and $\Phir''$ are bounded functions, we can estimate each term in \eqref{eq.commut} thanks to the $\sL^2(K)$ estimate given in  \eqref{eq.unif1} and \eqref{eq.unif2}  and we get
\begin{equation} \label{eq.commutateur}
 \left\| \left[h^2 D_{s}^2,\re^{\Phir/h}\right]u\right\|_{\sL^2(K)} = {\cal O}(h^\infty).
\end{equation}
From \eqref{eq.Dsint}, \eqref{eq.commutateur} and \eqref{eq.sumDS2}, we get the following estimate
\begin{equation} \label{eq.DSext}
 \| h^2 D_{s}^2(\re^{\Phir/h} u) \|_{\sL^2(K)} = {\cal O}(h^\infty).
\end{equation}
From Sobolev embedding, we deduce from \eqref{eq.DSext} and \eqref{eq.hDSexpu9} that
\begin{equation} \label{eq.DSlinf}
 \| h D_{s}(\re^{\Phir/h} u) \|_{\sL^\infty(K)} = {\cal O}(h^\infty).
\end{equation}
Now doing again the commutation between $hD_s$ and $\re^{\Phir/h}$ gives
\begin{equation} \label{eq.DSlinf2}
 \|\re^{\Phir/h} h D_{s} u \|_{\sL^\infty(K)}  \leq \| h D_{s}(\re^{\Phir/h} u) \|_{\sL^\infty(K)} + \| \Phir' \re^{\Phir/h} u \|_{\sL^\infty(K)}.
\end{equation}
Using then \eqref{eq.DSlinf} for the term with the derivative, the fact that $\Phir'$ is bounded and \eqref{eq.unif1} in the $\sL^\infty(K)$ sense, we get
\begin{equation} \label{eq.DSlinf3}
 \|\re^{\Phir/h} h D_{s} u \|_{\sL^\infty(K)}  = {\cal O}(h^\infty).
\end{equation}
The proof of the $\sL^\infty(K)$ estimate in \eqref{eq.unif2} is complete, and so is the proof of Proposition~\ref{uniform}.
\end{proof}
\begin{remark} \label{minuit}
The estimate given by Proposition \ref{uniform} is crucial and will be used in particular to get an estimate at the points $\pm \pi/2$ in Section \ref{Sec.Comput}.
\end{remark}

\subsection{From one well to the other}
In this section we explain how to transfer the informations for the well configuration $s=0$ to the one of $s=\pi$. In the following we index by $\le$ the quantities, operators, quasimodes, etc. related to the left-hand side well whose coordinate is $s=\pi$.\\
Let $\Ble{\rho}:=\cB(\pi,\rho)= (\pi-\rho,\pi+\rho),$ for any $\rho\in(0,\pi)$.
The Dirichlet realization of $(hD_{s}+\xi_{0})^2+V(s)$ on $\sL^2(\Iel,\dx s)$ is denoted $\Ll$. \\
Let us consider the transform $U$ defined by
\begin{equation}\label{U}
U(f)(s) = \overline{f(\pi-s)}.
\end{equation}
For any $\rho\in(0,\pi]$, the application $U$ defines an anti-hermitian unitary transform from $\sL^2(\Bri{\rho},\dx s)$ onto $\sL^2(\Ble{\rho},\dx s)$. According to Assumption \ref{V} about the symmetry of $V$, the two operators $\Lr$ and $\Ll$ are unitary equivalent:
\begin{equation}\label{eq.conj-op}
\Ll=U\Lr U^{-1}.
\end{equation}
Thus they have the same spectrum and $\lar$ is the first common eigenvalue. The eigenfunctions of $\Ll$ are obviously deduced from those of $\Lr$ thanks to the unitary transform $U$. We let $\tphil=U\tphir$. Then the function $\tphil$ is a positive $\sL^2$-normalized eigenfunction of $\tLl$ (the Dirichlet realization of $h^2D_{s}^2+V$ on $\sL^2(\Iel,\dx s)$) associated with $\lar$. Thus we have
$$\tLl \tphil = (h^2 D_{s}^2+V)\tphil = \lar \tphil\qquad \mbox{ on }\Iel.$$
The function $\phil$ defined on $\Iel$ by
\begin{equation}\label{def.phil}
\phil=U\phir,
\end{equation}
is an eigenfunction of $\Ll$ associated with $\lar$ and satisfies
\begin{equation}\label{phil}
\phil(s) = \re^{i\frac{\xi_{0} \pi}h} \re^{-i\frac{\xi_{0} s}h} \tphil(s),\qquad \forall s\in \Iel.
\end{equation}

\section{Double wells and interaction matrix}\label{Sec.doublewell}
\subsection{Estimates of Agmon}
In this section, we discuss the estimates of Agmon in the double well situation. These global estimates have a similar proof as in Proposition \ref{Agmon}. From now on, $\Phi$ will denote the global Agmon distance
$$
\Phi(s) = \min (\Phir(s), \Phi_{\le}(s)),
$$
with the Agmon distances defined as in \eqref{Phi0} by
\begin{equation}\label{eq.PhirPhil}
\Phir(s)=\int_{[0,s]} \sqrt{V(\sigma)} \dx \sigma, \; \forall s\in \Bri{\pi}
\quad\mbox{ and }\quad
\Phi_{\le}(s)=\int_{[\pi,s]} \sqrt{V(\sigma)} \dx \sigma, \;\forall s\in \Ble{\pi}.
\end{equation}
The function $\Phi$ is Lipschitzian and satisfies the eikonal equation $|\Phi'|^2=V$.
\begin{proposition}\label{Agmong} Let us consider the $\rho$-neighborhood of the wells on $\S1$ identified with $\R/ 2\pi \Z$
$$
\widehat\cB(\rho) = \Bri{\rho} \cup \Ble{\rho}.
$$
For all $\varepsilon \in(0,1)$, $C_{0}>0$, there exist positive constants $h_{0}, A, c, C$ such that, for all $h\in (0,h_{0})$, $z\in[0,C_{0}h]$ and $u\in {\cal C}^\infty (\S1)$,
\begin{equation}\label{Agmong-L2}
ch\| \re^{\sqrt{1-\varepsilon }\Phi/h}u\|_{\sL^2(\S1)}
\leq\| \re^{\sqrt{1-\varepsilon })\Phi/h}(\Lh -z)u\|_{\sL^2(\S1)}+ Ch\|u\|_{\sL^2(\widehat\cB(Ah^{1/2}))},
\end{equation}
and
\begin{equation}\label{Agmong-H1}
\left\| (h D_{s} + \xi_0) \left(\re^{\sqrt{1-\varepsilon }\Phi/h} u\right)\right\|^2_{\sL^2(\S1)}
\leq \frac{C}{h}\| \re^{\sqrt{1-\varepsilon }\Phi/h}(\Lh -z)u\|^2_{\sL^2(\S1)}+ Ch\|u\|_{\sL^2(\widehat\cB(Ah^{1/2}))}^2.
\end{equation}
\end{proposition}
\begin{proof}
For $\varepsilon \in(0,1)$, we let $\Phi_{\varepsilon }=\sqrt{1-\varepsilon }\Phi$. We apply Lemma \ref{localization} with $P = hD_s +\xi_0$, $Q= \re^{\Phi_{\varepsilon }/h}$, and use that $\Phi$ is Lipschitzian. After an integration by parts, we obtain
\[
\Re\int_{\S1} (hD_s+ \xi_0)^2 u\, \re^{2\Phi_{\varepsilon }/h}u\dx s
= \int_{\S1} \left|\left(h D_{s}+\xi_0\right) \left(\re^{\Phi_{\varepsilon }/h}u\right)\right|^2\dx s
- \int_{\S1} |\Phi_{\varepsilon }'|^2 \re^{2\Phi_{\varepsilon }/h} |u|^2\dx s.
\]
Adding the electric potential $V$ and recalling that $\Lh = (hD_s+ \xi_0)^2 + V$, we get
\begin{align*}
\int_{\S1} \left|(h D_{s}+\xi_0)\left(\re^{\Phi_{\varepsilon }/h}u\right)\right|^2\dx s
+\int_{\S1} \left(V-|\Phi_{\varepsilon }'|^2\right) \re^{2\Phi_{\varepsilon }/h} |u|^2\dx s
&= \Re\int_{\S1}\Lh u\, \re^{2\Phi_{\varepsilon }/h}u\dx s \\
&\leq \left\| \re^{\Phi_{\varepsilon }/h}\Lh u\right\|  \left\| \re^{\Phi_{\varepsilon }/h}u\right\|,
\end{align*}
so that
$$
\int_{\S1} \left|(hD_{s}+ \xi_0)\left(\re^{\Phi_{\varepsilon }/h}u\right)\right|^2\dx s+\int_{\S1} \varepsilon  V \re^{2\Phi_{\varepsilon }/h}{|u|^2}\dx s
\leq \left\| \re^{\Phi_{\varepsilon }/h}\Lh u\right\|  \left\| \re^{\Phi_{\varepsilon }/h}u\right\|.
$$
The rest of the proof is identical to the one of Proposition \ref{Agmon}, using again the non degeneracy of the minima of $V$ at $s=0$ and $s=\pi$ as in the proof of Proposition~\ref{prop.poids}. Then we get \eqref{Agmong-H1}.
\end{proof}
As a direct consequence of Proposition \ref{Agmong} with $u = \varphi$ and $z= \lambda$, we get 
\begin{corollary} \label{csqagmon}
For all $\varepsilon \in(0,1)$, there exist $C>0$ and $h_{0}>0$ such that, for $h\in(0,h_{0})$ and $\varphi$ an eigenfunction of $\Lh$ associated with $\lambda={\cal O}(h)$,
$$
\|\re^{\sqrt{1-\varepsilon }\Phi/h} \varphi\|_{\sL^2(\S1)}\leq C\|\varphi\|_{\sL^2(\S1)} \qquad\mbox{ and}\qquad
\|h D_s (\re^{\sqrt{1-\varepsilon }\Phi/h} \varphi )\|_{\sL^2(\S1)} \leq C\|\varphi\|_{\sL^2(\S1)}.
$$
\end{corollary}

\subsection{Rough estimates on the spectrum}
The main purpose of this article is to get an exponentially precise description of the lowest eigenvalues of $\Lh$. For this we use the one well unitary equivalent operators $\Lr$ and $\Ll$ defined respectively on $\Ier$ and $\Iel$. Let us consider the quadratic approximation of $\tLr$ defined on $\R$ by
$$
h^2D_{s}^2 + \frac{1}{2} V''(0) s^2.
$$
From a direct and standard analysis, we know that its spectrum is discrete, made of the simple eigenvalues $(2j+1) \kappa h$ for  $j \in \mathbb{N}$. In particular,  $\kappa h$ is a single eigenvalue in the interval $\Ih=(-\infty, 2 \kappa h )$. By quadratic approximation, we know that for any fixed  $\eta$, $\Lr$ has only a single eigenvalue $\lar$ in $\Ih $ satisfying
\begin{equation}\label{eq.lhkappa}
\lar = \kappa h + {\cal O}(h^{3/2}),
\end{equation}
since the eigenvalues are of type
\begin{equation}\label{eq.lhkappa2}
 (2j+1) \kappa h + {\cal O}(h^{3/2}), \qquad j \geq 0.
\end{equation}
In order to estimate the first two eigenvalues of the full operator $\Lh$ on $\S1$, which will appear to be very close to $\lar$ and the only ones in $\Ih$, we need to write the matrix of $\Lh$ on an appropriate invariant two dimensional subspace. For this we need to extend on $\S1$ the quasimodes built in the simple well cases.
\begin{notation} \label{notation} 
We will use the following conventions and notation:
\begin{enumerate}[(i)]
\item We identify functions on $\S1$ and $2\pi$-periodic functions of the variable $s \in \R$. We also extend by $0$ on $\S1 \setminus\Ier$ the functions $\chir$ and  $\phir$ and by $0$ on $\S1 \setminus\Iel$ the functions $\chil$ and  $\phil$. 
\item We index by $\alpha$ and $\beta$ the points $\ri$ and $\le$, and identify $\ri$ with $0$ and $\le$ with $\pi$ on $\S1$. For convenience, we also denote by $\bar{\alpha}$ the complement of $\alpha$ in $\{\ri,\le\}$.
\item for a given function $f$, we say that a function is ${\widetilde{\cal O}}(\re^{-f/h})$ if, for all $\varepsilon>0$, $\eta > 0$, it is ${\cal O}(\re^{(\varepsilon + \gamma(\eta) -f)/h})$, where $\lim_{\eta \rightarrow 0}\gamma(\eta) =0$ (see \cite{HelSj84,HelSj85,DiSj99}).
\end{enumerate}
\end{notation}

\begin{definition}
We introduce two quasimodes $\fr$ and $\fl$ defined on $\S1$ by
\begin{equation} \label{extension}
 \fr =\chir \phir
\qquad\mbox{ and }\qquad
\fl=\chil \phil,
\end{equation}
with
\begin{equation}\label{def.chil}
\chil=U\chir.
\end{equation}
\end{definition}
We have in particular $\fl=U\fr$. Since we want to compare the operators $\Lh$ and $\La$, we first compute $\Lh \fa$. 
\begin{lemma} \label{lemresta}
Let us denote, for $\alpha \in \{\le, \ri\}$,
\begin{equation} \label{defresta}
 \resta = (\Lh -\lar) \fa = (\La -\lar) \chia \phia = [\La, \chia] \phia.
\end{equation}
For $\eta$ sufficiently small, we have
\begin{enumerate}[{\rm(i)}]
\item\label{lemrestai} $\resta(s) = {\widetilde{\cal O}}( \re^{- \sfS /h})$,
\item\label{lemrestaibis} $\langle \resta,\fa\rangle = {\widetilde{\cal O}}( \re^{-2\sfS /h})$ and $\langle \resta,\fb\rangle = {\widetilde{\cal O}}( \re^{-\sfS /h})$  for $\alpha\neq \beta$,
\item\label{lemrestaii}$\langle \fa,\fa\rangle = 1 + {\widetilde{\cal O}}( \re^{-2\sfS /h})$
and $\langle \fa,\fb\rangle = {\widetilde{\cal O}}( \re^{-\sfS /h})$  for $\alpha\neq \beta$,
\item Let us introduce the finite dimensional  vectorial space ${\mathcal F} = \mathsf{span} \{\fr, \fl\}$.
Then, for $h$ small enough, $ \dim {\mathcal F} = 2$.
\end{enumerate}
\end{lemma}

\begin{proof}
\begin{enumerate}[{\rm(i)}]
\item Thanks to Corollary \ref{csqagmon}, we get in $\sL^\infty(\S1)$ and $\sL^2(\S1)$ sense that, for all $\varepsilon >0$,
$$
\re^{ \sqrt{1-\varepsilon }\Phi_\alpha(s)/h}\resta(s) = {\cal O}(1).
$$
Since the support of $[\La, \chia]$ is included in $\cB_{\bar{\alpha}}( 2\eta)$, we get: 
\begin{equation} \label{estimrest0}
\resta(s) = \widetilde{{\cal O}}( \re^{- \sfS/h}).
\end{equation}
\item is a consequence of \eqref{lemrestai} and the location of the support of $\resta$.
\item We first recall, from Proposition \ref{Agmon} and Proposition \ref{prop.poids} \eqref{poids1}, that
\begin{equation}\label{Agmon-r}
\phia = {\widetilde{\cal O}}( \re^{- \Phi_\alpha/h}),
\end{equation}
in $\sL^2(\Iea)$ and $\sH^1(\Iea)$. According to Agmon estimates, this gives in particular 
\begin{equation} \label{fafb}
\langle \fa,\fa\rangle = 1 + {\widetilde{\cal O}}( \re^{-2\sfS /h}).
\end{equation}
For $\alpha \neq \beta$, using \eqref{Agmon-r}, the supports of $\chia$ and $\chib$ and since $\Phi_\alpha+\Phi_\beta \geq \sfS$, we get
$$\langle \fa,\fb\rangle = {\widetilde{\cal O}}( \re^{-\sfS /h}).$$
\item The previous estimates imply that $\dim{\mathcal F}=2$ for $h$ small enough.
\end{enumerate}
\end{proof}
In the following series of lemmas, we show that the first two eigenvalues are exponentially close to $\lar$ and are the only ones in $\Ih $.
\begin{lemma}
Let us define ${\mathcal G}=\mathrm{range}\left(\mathds{1}_{\Ih}(\Lh)\right)$. Then
$\dist( \Sp(\Lh), \lar) = {\widetilde{\cal O}}(\re^{-\sfS /h})$ and ${\mathsf{dim}}\ {\mathcal G} \geq 2$.
\end{lemma}
\begin{proof}
This is a consequence of the spectral theorem. Indeed, using Lemma \ref{lemresta}, we get
$$\forall u\in {\mathcal F},\qquad \|(\Lh-\lar)u\|={\widetilde{\cal O}}( \re^{-\sfS /h})\|u\|.$$
This achieves the proof since $\dim{\mathcal F}=2$.
\end{proof}
Now we can prove the following.
\begin{lemma}\label{lem.SpLh}
We have
\begin{enumerate}[\rm(i)]
\item\label{specintermediaire} $ \langle (\Lh - \lar) u,u\rangle \geq \kappa h \|u\|^2$,  for all $u \in {\mathcal G}^\perp$,
\item\label{dimG} $\dim{\mathcal G} = 2$,
\item\label{SpLh} $ \Sp(\Lh) \cap \Ih  \subset [\lar- {\widetilde{\cal O}}(\re^{-\sfS /h}), \lar + {\widetilde{\cal O}}(\re^{-\sfS /h})]$.
\end{enumerate}
\end{lemma}

\begin{proof}
\begin{enumerate}[\rm(i)]
\item We use again a localization formula and consider a partition of unity $(\widetilde{\chi}_\le,\widetilde{\chi}_\ri)$ such that
$$
\widetilde{\chi}_\le^2 + \widetilde{\chi}_\ri^2 = 1\quad\mbox{ on }\S1,
$$
where $\widetilde{\chi}_\le = U \widetilde{\chi}_\ri$ and $\widetilde{\chi}_\ri$ is supported in $\Bri{3\pi/2}$, equal to 1 in  $\Bri{\pi/2}$. Writing the \enquote{IMS} formula, we deduce that, for $u \in {\mathcal F}^\perp$,
\begin{align*}
\langle (\Lh - \lar) u,u\rangle
& = \sum_{\alpha \in\{\le,\ri\}} \left\langle (\Lh - \lar) \widetilde{\chi}_\alpha u,\widetilde{\chi}_\alpha u\right\rangle + {\cal O}(h^2) \|u\|^2 .
\end{align*}
Let $\Pi_\alpha$ be the orthogonal projection on $\phia$, then
$$
 \widetilde{\chi}_\alpha u - \Pi_\alpha \widetilde{\chi}_\alpha u \in \langle \phia\rangle^\perp.
$$
With $\kappa$ defined in \eqref{lambda0}, we get
\begin{align} \label{endessous}
\left\langle (\Lh - \lar) u,u\right\rangle
& = \hspace{-.3cm}\sum_{\alpha \in \{\le,\ri\}} \hspace{-.2cm}\langle (\La - \lar) (\widetilde{\chi}_\alpha u - \Pi_\alpha \widetilde{\chi}_\alpha u) ,(\widetilde{\chi}_\alpha u- \Pi_\alpha \widetilde{\chi}_\alpha u)\rangle + {\cal O}(h^2) \|u\|^2 \nonumber\\
& \geq \sum_{\alpha \in \{\le,\ri\}} 2 \kappa { h} \left\| \widetilde{\chi}_\alpha u - \Pi_\alpha \widetilde{\chi}_\alpha u\right\|^2 + {\cal O}(h^{3/2}) \|u\|^2,
\end{align}
from \eqref{eq.lhkappa} and \eqref{eq.lhkappa2}.\\
Let us now check that there exists $c >0$ (uniform in $\eta$) such that
\begin{equation} \label{projpetit}
\left\|\Pi_\alpha \widetilde{\chi}_\alpha u\right\| = {\cal O}(\re^{-c/h}).
\end{equation}
For this we introduce new cut-off functions ${\widehat{\chi}}_\alpha$ such that $\widetilde{\chi}_\alpha \prec {\widehat{\chi}}_\alpha \prec \chi_\alpha$, that is to say $\supp\widetilde{\chi}_{\alpha} \subset\{\widehat{\chi}_{\alpha}\equiv1\}$ and $\supp\widehat{\chi}_{\alpha} \subset\{\chia\equiv1\}$. Thanks to the condition on the support, we have
$$
{\widehat{\chi}}_\alpha u \perp \fa.
$$
Since $\fa=\phia$ on the support of $\widetilde{	\chi}_{\alpha}$, we check that
\begin{align}
\left\| \Pi_\alpha \widetilde{\chi}_\alpha u \right\|
& = \left| \langle \widetilde{\chi}_\alpha u, \phia\rangle\right| = \left| \langle \widetilde{\chi}_\alpha u, \fa\rangle \right|
 = \left| \langle (\widetilde{\chi}_\alpha - {\widehat{\chi}}_\alpha) u, \fa\rangle\right| \nonumber\\
& \leq \left\| (\widetilde{\chi}_\alpha - {\widehat{\chi}}_\alpha) \fa\right\| \|u\| = {\cal O}(\re^{-c/h}) \|u\|, \label{eq.Pialpha}
\end{align}
thanks to Corollary \ref{csqagmon}. This gives \eqref{projpetit}. From \eqref{endessous} and \eqref{eq.Pialpha}, we infer
\begin{align*}
\langle (\Lh - \lar) u,u\rangle
& \geq \sum_{\alpha \in \{\le,\ri\}} 2 \kappa h \left\| \widetilde{\chi}_\alpha u\right\|^2 + {\cal O}(h^{3/2}) \|u\|^2 \\
& \geq   \kappa h \|u\|^2,
\end{align*}
for $h$ small enough. This gives \eqref{specintermediaire}. 
\item Now using again the first inequality in the preceding computation also gives
\begin{align*}
\langle \Lh u,u\rangle
& \geq \sum_{\alpha \in \{\le,\ri\}} 2 \kappa h \left| \widetilde{\chi}_\alpha u\right\|^2 + \lambda(h) \|u\|^2 +  {\cal O}(h^{3/2}) \|u\|^2 \\
& \geq 2   \kappa h \|u\|^2, 
\end{align*}
from \eqref{eq.lhkappa} and for $h$ small enough. From the min-max principle and since  $\{\fl, \fr\}$ is a free family, we get $\dim{\mathcal G} \leq 2$ and we deduce \eqref{dimG}.
\item Eventually using  Lemma~\ref{lemresta} \eqref{lemrestai}, we get \eqref{SpLh} and the proof is complete.
\end{enumerate}
\end{proof}

\subsection{Precised estimates about quasimodes and eigenfunctions}
In this section we give precise estimates of the quasimodes $\fa$ and their projections on the spectral subspaces $\ga = \Pi \fa$ where $\Pi$ denotes the projection on ${\mathcal G}$. Let us first estimate the difference between $\fa$ and $\ga$.
\begin{lemma} \label{vafa}
We have $\fa - \ga = {\widetilde{\cal O}}(\re^{-\sfS /h})$ in $\sL^2(\S1)$ and $\sH^1(\S1)$.
\end{lemma}

\begin{proof}
We write
$$
(\Lh-\lar)(\fa-\ga)  = (\Lh-\lar)\fa - (\Lh-\lar)\ga.
$$
The first term is ${\widetilde{\cal O}}(\re^{-\sfS /h})$ from Lemma \ref{lemresta} \eqref{lemrestai}. The second is ${\widetilde{\cal O}}(\re^{-\sfS /h})$ from the exponential localization in Lemma \ref{lem.SpLh} \eqref{SpLh}. We therefore get in $\sL^2(\S1)$
$$
(\Lh-\lar)(\fa-\ga) = {\widetilde{\cal O}}(\re^{-\sfS /h}).
$$
Since $\fa-\ga \in {\mathcal G}^\perp$, we can use Lemma \ref{lem.SpLh}  \eqref{specintermediaire} and the spectral theorem to conclude that
$$
\fa-\ga = {\widetilde{\cal O}}(\re^{-\sfS /h}) \qquad \mbox{\rm in } \sL^2(\S1).
$$
By using the two preceding estimates, we get the result in $\sH^1(\S1)$.
\end{proof}
The following obvious lemma will be convenient in the following.
\begin{lemma}\label{Py}
Let $(\sH,\langle\cdot,\cdot\rangle)$ be a Hilbert space and $\Pi\in\mathcal{L}(\sH)$ be an orthogonal projection. Then, for all $u,v\in\sH$, we have
$$\langle u,v \rangle=\langle \Pi u,\Pi v \rangle+\langle(\mathsf{Id}-\Pi) u,(\mathsf{Id}-\Pi)v \rangle.$$
\end{lemma}

\begin{lemma} \label{ps}
Let us define  the matrix $\mathsf T=(\mathsf T_{\alpha,\beta})_{\alpha,\beta\in\{\le,\ri\}}$ with $\mathsf T_{\alpha,\beta} = \langle\fa,\fb\rangle $ if $\alpha \neq \beta$ and $0$ otherwise. Then $\mathsf T = {\widetilde{\cal O}}(\re^{-\sfS /h})$ and we have
\begin{enumerate}[{\rm(i)}]
\item\label{ps-ii} $\left(\langle\fa,\fb\rangle\right)_{\alpha,\beta\in\{\le,\ri\}} = \mathsf{Id} + \mathsf T + {\widetilde{\cal O}}(\re^{-2\sfS /h})$,
\item\label{ps-i} $\langle\ga,\gb\rangle = \langle\fa,\fb\rangle + {\widetilde{\cal O}}(\re^{-2\sfS /h})$,
\item\label{ps-iii} $\left(\langle\ga,\gb\rangle\right)_{\alpha,\beta\in\{\le,\ri\}} = \mathsf{Id} + \mathsf T + {\widetilde{\cal O}}(\re^{-2\sfS /h})$.
\end{enumerate}
\end{lemma}
\begin{proof}
The fact that $\mathsf T  = {\widetilde{\cal O}}(\re^{-\sfS /h})$ and \eqref{ps-ii} follow from Lemma \ref{lemresta} \eqref{lemrestaii}.
\eqref{ps-i} is a consequence of Lemma \ref{Py} and Lemma \ref{vafa}. 
\eqref{ps-iii} is then obvious.
\end{proof}

\subsection{Interaction matrix}
From Lemma \ref{ps} \eqref{ps-iii}, the basis $(\gl, \gr)$ is quasi orthonormal but not exactly orthonormal.  Therefore we introduce the new basis $\mathsf{g} = g \mathsf G^{-1/2}$, where $\mathsf G$ is the Gram-Schmidt matrix $(\langle\ga, \gb\rangle)_{\alpha,\beta\in\{\le,\ri\}}$ and $g$ the row vector $(\gl,\gr)$. The basis $\mathsf{g}$ is orthonormal since
$$
(\langle\gna, \gnb\rangle)_{\alpha,\beta\in\{\le,\ri\}} = {^t}{\mathsf G}^{-1/2} (\langle \ga, \gb\rangle)_{\alpha,\beta\in\{\le,\ri\}}\mathsf G^{-1/2} 
=\mathsf G^{-1/2} \mathsf G \mathsf G^{-1/2} =\mathsf{Id}.
$$

\begin{proposition}\label{prop.interaction}
The matrix $\mathsf M$ of the restriction to $\Lh$ in the basis $\mathsf{g}$ is given by
$$
\mathsf M := \left(\langle\Lh\mathsf{g}_\alpha,\mathsf{g}_\beta\rangle\right)_{\alpha,\beta\in\{\le,\ri\}}= \mathsf D + \mathsf W + {\widetilde{\cal O}}(\re^{-2\sfS /h}),
$$
where
\begin{enumerate}[(a)]
\item\label{prop.interactioni}  $\mathsf D = \lar \mathsf{Id}$,
\item\label{prop.interactionii} the \enquote{interaction matrix}
$\mathsf W=(w_{\alpha,\beta}(h))_{\alpha,\beta\in\{\le,\ri\}}$ is defined, recalling \eqref{defresta}, by
$$
w_{\alpha,\beta}(h) = \langle\resta, \fb\rangle \quad \mbox{ if }\alpha \neq \beta,\qquad\mbox{ and }\quad0\mbox{ otherwise}.
$$
\end{enumerate}
In particular, the gap between the two first eigenvalues, denoted by $\lambda_{1}(h)$ and $\lambda_{2}(h)$, of $\Lh$ (or of $M$) satisfies
\begin{equation} \label{splitth}
\lambda_2(h)-\lambda_1(h) = 2 |w_{\le,\ri}(h)|+ {\widetilde{\cal O}}(\re^{-2\sfS /h}).
\end{equation}
\end{proposition}
For the proof of Proposition \ref{prop.interaction} we begin by two lemmas. First, we notice that $\mathsf W$ is indeed an Hermitian matrix by using the symmetries of our constructions.
\begin{lemma}
The matrix $\mathsf W$ is Hermitian.
\end{lemma}
\begin{proof}
By definition, we have $w_{\alpha,\alpha}(h) = 0$ for $\alpha \in \{\ri, \le\}$ and
$$w_{\le,\ri}(h)=\left\langle [\Ll, \chil] \phil,\chir\phir\right\rangle.$$
By using \eqref{eq.conj-op}, \eqref{def.phil} and \eqref{def.chil}, we deduce that
\begin{align*}
w_{\le,\ri}(h)&=\left\langle [U\Lr U^{-1},U \chir] U\phir,U^{-1}\left(\chil\phil\right)\right\rangle\\
& = \left\langle  U\Lr U^{-1} (U \chir U\phir) -  U \chir U\Lr U^{-1} (U\phir),U^{-1}\left(\chil\phil\right)\right\rangle\\
& = \left\langle  U\Lr (\chir \phir) -  U \chir U\Lr (\phir),U^{-1}\left(\chil\phil\right)\right\rangle\\
& = \left\langle  U\left(\Lr (\chir \phir) -  \chir \Lr (\phir)\right),U^{-1}\left(\chil\phil\right)\right\rangle\\
&=\left\langle U\left([\Lr , \chir] \phir\right),U^{-1}\left(\chil\phil\right)\right\rangle\\
&=\overline{\left\langle [\Lr , \chir] \phir,\chil\phil\right\rangle}  = \overline{w_{\ri,\le}(h)},
\end{align*}
since $U$ is anti-hermitian. 
\end{proof}
Then, we write the matrix of $\Lh$ in the quasi orthonormal basis $g$.
\begin{lemma} \label{lhps}
We have
\begin{enumerate}[\rm (i)]
\item\label{tintin} $\langle\Lh\ga,\gb\rangle = \langle\Lh\fa,\fb\rangle + {\widetilde{\cal O}}(\re^{-2\sfS /h})$,
\item\label{tintin2} $\left(\langle\Lh \fa,\fb\rangle\right)_{\alpha,\beta\in\{\le,\ri\}} = \mathsf D + \mathsf D\mathsf T  + \Ws + {\widetilde{\cal O}}(\re^{-2\sfS /h})$,
\item\label{tintin3} $\left(\langle\Lh\ga,\gb\rangle\right)_{\alpha,\beta\in\{\le,\ri\}} = \mathsf D + \mathsf D\mathsf T  + \Ws + {\widetilde{\cal O}}(\re^{-2\sfS /h})$.
\end{enumerate}
\end{lemma}
\begin{proof} 
\begin{enumerate}[\rm (i)] 
\item With Lemma \ref{Py}, we get
$$ \langle\Lh\fa,\fb\rangle - \langle\Lh\ga,\gb\rangle = \langle\Lh(\fa- \ga),\fb-\gb\rangle.$$
From Lemma \ref{vafa} applied in $\sH^1$, we get directly that
$$
\langle\Lh\ga,\gb\rangle - \langle\Lh\fa,\fb\rangle = {\widetilde{\cal O}}(\re^{-2\sfS /h}).
$$
\item We can write
$$
\langle \Lh\fa,\fb\rangle = \lar \langle\fa,\fb\rangle + \langle\resta,\fb\rangle.
$$
The result follows from the definition of $\mathsf D$, $\mathsf W$, Lemma~\ref{lemresta} \eqref{lemrestaibis} and Lemma \ref{ps} \eqref{ps-ii}.
\item This is a direct consequence of \eqref{tintin} and \eqref{tintin2}.
\end{enumerate}
\end{proof}

\begin{proofof}{Proposition \ref{prop.interaction}}
Since $\mathsf{g} = g \mathsf G^{-1/2}$, we directly get 
$$
\mathsf M = \mathsf G^{-1/2} (\langle\Lh \ga, \gb\rangle)_{\alpha,\beta\in\{\le,\ri\}} \mathsf G^{-1/2}.
$$
Recall that Lemma \ref{ps} \eqref{ps-iii} gives $\mathsf G = \mathsf{Id} + \mathsf T + {\widetilde{\cal O}}(\re^{-2\sfS /h})$. Using Lemma \ref{lhps} \eqref{tintin3}, we get  
\begin{equation*}
\begin{split}
\mathsf M 
& = \big(\mathsf{Id} + \mathsf T + {\widetilde{\cal O}}(\re^{-2\sfS /h})\big)^{-1/2} \big( \mathsf D + \mathsf D\mathsf T  + \Ws + {\widetilde{\cal O}}(\re^{-2\sfS /h})\big) \big(\mathsf{Id} + \mathsf T + {\widetilde{\cal O}}(\re^{-2\sfS /h})\big)^{-1/2} \\
& = \big(\mathsf{Id} -  \tfrac{1}{2} \mathsf T + {\widetilde{\cal O}}(\re^{-2\sfS /h})\big) \big(\mathsf D + \mathsf D\mathsf T  + \Ws + {\widetilde{\cal O}}(\re^{-2\sfS /h})\big) \big(\mathsf{Id} -  \tfrac{1}{2} \mathsf T + {\widetilde{\cal O}}(\re^{-2\sfS /h})\big) \\
& = \mathsf D + \mathsf D\mathsf T + \mathsf W - \tfrac{1}{2}\mathsf  T \mathsf D - \tfrac{1}{2} \mathsf D \mathsf T + {\widetilde{\cal O}}(\re^{-2\sfS /h}) \\
& = \mathsf D + \mathsf W + {\widetilde{\cal O}}(\re^{-2\sfS /h}), 
\end{split}
\end{equation*}
where we used that $\mathsf W = {\widetilde{\cal O}}(\re^{-\sfS /h})$ from Lemma~\ref{lemresta} \eqref{lemrestaibis}, $\mathsf T = {\widetilde{\cal O}}(\re^{-\sfS /h})$ from Lemma \ref{ps}, and that $\mathsf D$ and $\mathsf T$ commute by definition of $\mathsf D$. The spectrum of the $2\times2$ matrix $\mathsf D+\mathsf W$ is explicit and we deduce \eqref{splitth}. This completes the proof of Proposition \ref{prop.interaction}.
\end{proofof}

\section{Computation of the interaction\label{Sec.Comput}}
This section is devoted to computation of $w_{\le,\ri}(h)$ introduced in Proposition \ref{prop.interaction} and to the proof of Theorem \ref{th.gap}.
\subsection{Expression of the interaction coefficient}
First, we notice that using  \eqref{phil} and the $2\pi$-periodic extensions (see Notation \ref{notation}), the function $\phil$ writes on $(-\pi,\pi)$
\begin{equation}\label{phil-morceaux}
\phil(s) = \begin{cases}
 \re^{i\frac{\xi_{0} \pi}h} \re^{-i\frac{\xi_{0} s}h} \tphil(s),& \forall s\in (\eta,\pi),\\
\re^{-i\frac{\xi_{0} \pi}h}\re^{-i\frac{\xi_{0} s}h} \tphil(s),& \forall s\in (-\pi,-\eta),\\
0,&\forall s\in [-\eta,\eta].
\end{cases}
\end{equation}
By integration by parts, we have 
\begin{eqnarray*}
w_{\le,\ri}(h)
&=& -h^2\int_{\S1}\chil''\phil \overline{\phir} \dx s +\frac{2h}i\int_{\S1}\chil'(hD_{s}+\xi_{0})\phil\ \overline{\phir} \dx s\\
&=& h^2\int_{\S1}\chil'\left(\phil \overline{\phir}'-\phil' \overline{\phir}\right) \dx s
 +\frac{2h\xi_{0}}i \int_{\S1}\chil' \phil\ \overline{\phir} \dx s\\
&=& -ih\int_{\S1}\chil'\left( \phil\ \overline{(hD_{s}+\xi_{0})\phir}+(hD_{s}+\xi_{0})\phil\ \overline{\phir}\right) \dx s\\
&=& w_{\le,\ri}^{\mathsf u}+w_{\le,\ri}^{\mathsf d},
\end{eqnarray*}
with
\begin{align*}
w_{\le,\ri}^{\mathsf u}
&= -ih \int_{0}^\pi \chil'\left( \phil\ \overline{(hD_{s}+\xi_{0})\phir}+(hD_{s}+\xi_{0})\phil\ \overline{\phir}\right) \dx s\\
&=h^2\re^{i\frac{\xi_{0} \pi}h} \int_{0}^\pi \chil'\ \mathsf{Wronsk} \dx s,\\
w_{\le,\ri}^{\mathsf d}
&= -ih \int_{-\pi}^0 \chil'\left( \phil\ \overline{(hD_{s}+\xi_{0})\phir}+(hD_{s}+\xi_{0})\phil\ \overline{\phir}\right) \dx s\\
&= h^2\re^{-i\frac{\xi_{0} \pi}h} \int_{-\pi}^0 \chil'\ \mathsf{Wronsk} \dx s,
\end{align*}
where we have used \eqref{def.phir}, \eqref{phil-morceaux}, the fact that $\tphir$ and $\tphil$ are real valued and the notation
$$ \mathsf{Wronsk} = \tphil\ \tphir' -\tphil'\ \tphir.$$
Note that $\mathsf{Wronsk}$ is defined and constant on each of the two connected components of the support of $\chil'$, respectively included in
$(\eta,2\eta)$ and  $(-2\eta,-\eta)$ (modulo $2\pi$).
Also note that  
$$
\int_{0}^\pi\chil' \dx s =\int_{\eta}^{2\eta}\chil'  \dx s= \chil(2\eta) - \chil(\eta) = 1,
$$
according to the definition of $\chil$. Thus, since $\tphil=U\tphir$ and the functions are real valued, we can write
\begin{eqnarray*}
\mathsf{Wronsk}(s)
= \tphil \left(\frac\pi2\right)\ \tphir'\left(\frac\pi2\right) -\tphil'\left(\frac\pi2\right)\ \tphir\left(\frac\pi2\right)
= 2\tphir \left(\frac\pi2\right)\ \tphir'\left(\frac\pi2\right),\quad\forall s\in(0,\pi).
\end{eqnarray*}
In the same way,
$$\mathsf{Wronsk}(s)
= 2\tphir \left(-\frac\pi2\right)\ \tphir'\left(-\frac\pi2\right),\qquad\forall s\in(-\pi,0).$$
Consequently
\begin{equation} \label{sanssymetrie}
w_{\le,\ri}(h)
= 2h^2\left(\re^{i\frac{\xi_{0} \pi}h} \tphir \left(\frac\pi2\right)\ \tphir'\left(\frac\pi2\right)
-\re^{-i\frac{\xi_{0} \pi}h} \tphir \left(-\frac\pi2\right)\ \tphir'\left(-\frac\pi2\right)\right).
\end{equation}
In particular, if the potential $V$ is even so is $\tphir$ (whereas $\tphir'$ is odd) and we get
\begin{equation} \label{avecsymetrie}
w_{\le,\ri}(h)
= 4h^2\cos\left(\frac{\xi_{0} \pi}h\right)\ \tphir\left(\frac\pi2\right)\ \tphir'\left(\frac\pi2\right).
\end{equation}

\subsection{Proof of Theorems~\ref{th.gap} and \ref{th.gapsym}}
One of the consequence of Proposition \ref{uniform} (see also Remark \ref{minuit}) is that for any  compact  $K\subset \Ier$ and $N >0$, 
$$
\tphir = \BKWr + h^N {\cal O}( \re^{-\Phir/h}),
$$
in $\sL^\infty(K)$ and $\sW^{1,\infty}(K)$. Using the unitary transform $U$, we have
$$
2\Phir(\tfrac\pi2) =\sfS_{\up}\geq \sfS\qquad\mbox{ and }\qquad 2\Phir(-\tfrac\pi2) = \sfS_{\dow}\geq \sfS .
$$
Using \eqref{sanssymetrie}, this allows to write for all $N>0$
\begin{equation} \label{wlr1}
w_{\le,\ri}(h)
= 2h^2\left(\re^{i\frac{\xi_{0} \pi}h} \BKWr \left(\frac\pi2\right)\ \BKWr'\left(\frac\pi2\right)
-\re^{-i\frac{\xi_{0} \pi}h} \BKWr \left(-\frac\pi2\right)\ \BKWr'\left(-\frac\pi2\right)\right) + h^{N} {\cal O}(\re^{-\sfS /h}).
\end{equation}
We now use Lemma \ref{lem.QM} for computing this coefficient. We first write that
\begin{equation} \label{psipi2}
\BKWr \left(\frac\pi2\right) = h^{-1/4} \left( \frac{\kappa}{\pi}\right)^{1/4} \sqrt{\sfA_{\up}} \re^{-\sfS_{\up}/2h} (1+{\cal O}(h)),
\end{equation}
with
$$
\sfA_{\up} = \exp\left( - \int_{[0,\frac\pi2]} \frac{\partial_\sigma {\sqrt{V}} - \kappa}{\sqrt{V}} d\sigma\right),
$$
and
\begin{equation} \label{psiprimepi2}
\BKWr' \left(\frac\pi2\right) = h^{-5/4} \left( \frac{\kappa}{\pi}\right)^{1/4} \sqrt{\sfA_{\up}}\, \Phir'\left(\frac\pi2\right) \re^{-\sfS_{\up}/2h} (1+{\cal O}(h)).
\end{equation}
A similar expression is available for $\BKWl$ and its derivative at $-\pi/2$, with in particular
$$
\sfA_{\dow} = \exp\left( \int_{[-\frac\pi2, 0]} \frac{\partial_\sigma {\sqrt{V}} + \kappa}{\sqrt{V}} d\sigma\right).
$$
We take $N=2$ and use \eqref{wlr1}, \eqref{psipi2}, \eqref{psiprimepi2} and the fact that 
$$
\Phir'\left(\frac\pi2\right)=\sqrt{V\left(\frac\pi2\right)}
\qquad \mbox{ and }\qquad\Phir'\left(-\frac\pi2\right)=-\sqrt{V\left(-\frac\pi2\right)},
$$
to get 
\begin{equation*} \label{48}
w_{\le,\ri}(h)
= 2h^{1/2}\sqrt{ \frac{\kappa}{\pi}} \left(\re^{i\frac{\xi_{0} \pi}h}\sfA_{\up} \sqrt{V\left(\frac\pi2\right)} \re^{-\sfS_{\up}/h}
+\re^{-i\frac{\xi_{0} \pi}h}\sfA_{\dow} \sqrt{V\left(-\frac\pi2\right)} \re^{-\sfS_{\dow}/h} \right) + h^{3/2} {\cal O}(\re^{-\sfS /h}).
\end{equation*}
To deduce Theorem \ref{th.gap}, we use now splitting formula \eqref{splitth} in Proposition \ref{prop.interaction} and have to control the remainder. 
This can be done by taking $\varepsilon$ and $\eta$ small enough (see Notation \ref{notation}) so that ${\widetilde{\cal O}}(\re^{-2\sfS/h})=h^{3/2}{\cal O}(\re^{-\sfS/h})$.\\
Theorem~\ref{th.gapsym} is a direct consequence of Theorem~\ref{th.gap}.

\paragraph{Acknowledgments.}
This work was partially supported by the ANR (Agence Nationale de la Recherche), project {\sc Nosevol} n$^{\rm o}$ ANR-11-BS01-0019 and by the Centre Henri Lebesgue (program \enquote{Investissements d'avenir} -- n$^{\rm o}$ ANR-11-LABX-0020-01).

\def\cprime{$'$}

\end{document}